\documentclass[a4paper]{amsart}
\usepackage{amscd, amssymb, booktabs}
\usepackage{myown}
\title{A Recursive Relation for Bipartition Numbers}
\author{Yen-Chi Roger Lin}
\address{Department of Mathematics, National Taiwan Normal University, Taipei 116, Taiwan}
\email{yclin@math.ntnu.edu.tw}
\author{Shu-Yen Pan}
\address{Department of Mathematics, National Tsing Hua University, Hsinchu 300, Taiwan}
\email{sypan@math.nthu.edu.tw}


\date{\today}

\begin{document}

\begin{abstract}
We establish a recursive relation for the bipartition number $p_2(n)$
which might be regarded as an analogue of Euler's recursive relation for the partition number $p(n)$.
Two proofs of the main result are proved in this article.
The first one is using the generating function,
and the second one is using combinatoric objects (called ``symbols'') created by Lusztig for studying
representation theory of finite classical groups.
\end{abstract}

\maketitle
\tableofcontents

\section{Introduction}
For a positive integer $n$, let $p(n)$ (resp.~$p_2(n)$) denote the number of partitions
(resp.~bipartitions) of $n$.
By convention, we define $p(0)=p_2(0)=1$, and
$p(x)=p_2(x)=0$ if $x$ is not a non-negative integer.
For $n\geq 1$, Euler obtained the following recursive relation for $p(n)$:
\[
p(n)=\sum_{k\in\bbZ\smallsetminus\{0\}}(-1)^{k-1}p\bigl(n-\tfrac{1}{2}k(3k-1)\bigr)
\]
(cf.~\cite{andrews} corollary 1.8).
Our result in this article is the following recursive relation for $p_2(n)$:

\begin{thm}\label{0101}
For $n\geq 0$, we have
\[
p_2(n)=p\bigl(\tfrac{n}{2}\bigr)+\sum_{k\in\bbZ\smallsetminus\{0\}}(-1)^{k-1} p_2\bigl(n-k^2\bigr).
\]
\end{thm}

We will provide two proofs for the above recursive formula from different point of views.
In Section 2, the theorem is first proved by manipulating the respective generating functions.
In Section 3, we will derive the recursive relation from Lusztig's theory of ``symbols'' and the binomial theorem.

It is well known that the irreducible characters of the Coxeter  group of type $B_n$ is parametrized by the set of
bipartitions of $n$ (\cf.~\cite{Geck-Pfeiffer} \S 5.5),
just like the set of irreducible characters of the Coxeter group of type $A_n$ (i.e., the symmetric group $S_n$)
is parametrized by the set of partitions of $n$.
It is not surprising that Theorem~\ref{0101} has applications to the representation theory of finite classical groups.
Two consequences are given in Remark~\ref{0302} and Remark~\ref{0303}.

In the appendix, we provide a new proof of a congruence relation for the bipartition number $p_2(n)$ as a reference.

\section{Generating Functions}

\subsection{Partitions and bipartitions}
For a nonnegative integer $n$,
let $\calp(n)$ denote the set of partitions of $n$.
By convention, we define $\calp(0)$ to be the set of ``the empty partition'',
and $\calp(x)$ to be the empty set if $x$ is not a nonnegative integer.
Let $p(n)=|\calp(n)|$ where $|X|$ denotes the number of elements of a finite set $X$.

Let $\calp_2(n)$ denote the set of \emph{bi-partitions} of $n$, i.e.,
\[
\textstyle
\calp_2(n)=\left\{\,\sqbinom{\lambda}{\mu}\mid|\lambda|+|\mu|=n\,\right\}
\]
where $|\lambda|$ means the sum of parts of $\lambda$.
Let $p_2(n)=|\calp_2(n)|$.
We know that $p_2(0)=1$, and $p_2(x)=0$ if $x$ is not a nonnegative integer.
For a bipartition $\Sigma=\sqbinom{\lambda}{\mu}$,
its \emph{transpose} is given by $\Sigma^\rmt=\sqbinom{\mu}{\lambda}$.
A bipartition $\Sigma$ is called \emph{degenerate} if $\Sigma^\rmt=\Sigma$, it is called \emph{non-degenerate}, otherwise.
If $\Sigma\in\calp_2(n)$ is degenerate, then $\Sigma=\sqbinom{\lambda}{\lambda}$ for some $\lambda\in\calp(\tfrac{n}{2})$.
Thus, the number of degenerate bipartitions in $\calp_2(n)$ is equal to $p(\tfrac{n}{2})$.

\subsection{First proof of Theorem~\ref{0101}}

\begin{lem}\label{0204}
We have
\[
\sum_{n\in\bbZ} (-1)^n z^{n^2}
=\prod_{k=1}^\infty\frac{(1-z^k)^2}{1-z^{2k}}.
\]
\end{lem}
\begin{proof}
We start with the \emph{Jacobi's triple product identity} (\cf.~\cite{andrews} theorem 2.8)
\begin{equation}\label{0203}
\sum_{n\in\bbZ}q^{\frac{n(n+1)}{2}}z^n
=\prod_{k=1}^\infty (1+zq^k)(1+z^{-1}q^{k-1})(1-q^k).
\end{equation}
In both sides of (\ref{0203}), we replace $z$ by $z^{-1}$, and replace $q$ by $z^2$ to obtain
\[
\sum_{n\in\bbZ} z^{n^2}
=\prod_{k=1}^\infty(1+z^{2k-1})^2(1-z^{2k}).
\]
Then we replace $z$ by $-z$ to get
\[
\sum_{n\in\bbZ}(-1)^n z^{n^2}
=\prod_{k=1}^\infty(1 - z^{2k-1})^2(1-z^{2k})
=\prod_{k=1}^\infty(1 - z^{2k-1})(1-z^k).
\]

Now for the right-hand side,
we need the \emph{Euler's identity}: $p_\rmd(n)=p_\rmo(n)$
where $p_\rmd(n)$ is the number of partitions into distinct parts of $n$,
and $p_\rmo(n)$ is the number of partitions into odd parts.
In terms of generating functions, we have
\[
\prod_{k=1}^\infty(1 + z^k)
=\sum_{n=0}^\infty p_\rmd(n) z^n
=\sum_{n=0}^\infty p_\rmo(n) z^n
=\prod_{k=1}^\infty\frac{1}{1-z^{2k-1}}.
\]
Therefore,
\begin{align*}
\prod_{k=1}^\infty(1-z^{2k-1})(1-z^k)
= \prod_{k=1}^\infty\frac{1-z^k}{1 + z^k}
&= \prod_{k=1}^\infty\frac{(1-z^k)^2}{(1+z^k)(1-z^k)} \\
&= \prod_{k=1}^\infty\frac{(1-z^k)^2}{1-z^{2k}},
\end{align*}
and the lemma is proved.
\end{proof}

\begin{proof}[First Proof of Theorem~\ref{0101}]
Recall that
\[
\sum_{n=0}^\infty p(n)z^n
=\prod_{k=1}^\infty\frac{1}{1-z^k}\qquad\text{and}\qquad
\sum_{n=0}^\infty p_2(n)z^n
=\prod_{k=1}^\infty\frac{1}{(1-z^k)^2}.
\]
By Lemma~\ref{0204}, we have
\[
\left(\sum_{n=0}^\infty p_2(n)z^n\right)
\left(\sum_{k\in\bbZ}(-1)^k z^{k^2}\right)
=\sum_{n=0}^\infty p(n)z^{2n},
\]
which implies that
\[
\sum_{k\in\bbZ}(-1)^k p_2(n - k^2)
=p\bigl(\tfrac{n}{2}\bigr),
\]
i.e., the theorem is proved.
\end{proof}

\section{Lusztig's Symbols}

\subsection{Definition of the symbols}
In this subsection and the next subsection, we recall the notions of ``symbols'' by Lusztig.
The materials are modified from \cite{lg} \S 2, \S 3 (see also \cite{pan-finite-unipotent} \S 2).

A \emph{symbol} is an ordered pair
$\Lambda=\binom{A}{B}=\binom{a_1,a_2,\ldots,a_{m_1}}{b_1,b_2,\ldots,b_{m_2}}$
of two finite subsets $A,B$ of nonnegative integers.
The elements in each of $A$ and $B$ are written in decreasing order, i.e.,
$a_1>a_2>\cdots>a_{m_1}$ and $b_1>b_2>\cdots>b_{m_2}$.
The sets $A,B$ in $\Lambda$ are also denoted by $\Lambda^*,\Lambda_*$, respectively,
and called the \emph{first row} and the \emph{second row} of $\Lambda$.
An element in $\Lambda^*$ or $\Lambda_*$ is called an \emph{entry} of $\Lambda$.
A symbol $\Lambda'$ is called a \emph{subsymbol} of $\Lambda$ and denoted by $\Lambda'\subset\Lambda$
if $\Lambda'^*\subset\Lambda^*$ and $\Lambda'_*\subset\Lambda_*$.
For a subsymbol $\Lambda'\subset\Lambda$, we define the substraction
$\Lambda\smallsetminus\Lambda'$ to be $\binom{\Lambda^*\smallsetminus\Lambda'^*}{\Lambda_*\smallsetminus\Lambda'_*}$.
For two symbols $\Lambda,\Lambda'$, their
\emph{union} is defined by $\Lambda\cup\Lambda'=\binom{\Lambda^*\cup \Lambda'^*}{\Lambda_*\cup\Lambda'_*}$.

For a symbol $\Lambda=\binom{A}{B}=\binom{a_1,a_2,\ldots,a_{m_1}}{b_1,b_2,\ldots,b_{m_2}}$,
its \emph{rank} ${\rm rk}(\Lambda)$, \emph{defect} ${\rm def}(\Lambda)$  and \emph{transpose} $\Lambda^\rmt$ are defined by
\begin{align*}
{\rm rk}(\Lambda) &= \sum_{i=1}^{m_1}a_i+\sum_{j=1}^{m_2}b_j-\biggl\lfloor\biggl(\frac{|A|+|B|-1}{2}\biggr)^2\biggr\rfloor,\\
{\rm def}(\Lambda) &= |A|-|B|,\\
\Lambda^\rmt &=\binom{B}{A}
\end{align*}
where $\lfloor x\rfloor$ denotes the floor function of a real number $x$.
It is not difficult to check that ${\rm rk}(\Lambda)\geq\bigl\lfloor\bigl(\frac{{\rm def}(\Lambda)}{2}\bigr)^2\bigr\rfloor$
for any symbol $\Lambda$.
A symbol $\Lambda$ is called \emph{degenerate} if $\Lambda^\rmt=\Lambda$;
it is called \emph{non-degenerate}, otherwise.

On the set of symbols , we define an equivalence relation generated by
\[
\binom{a_1,a_2,\ldots,a_{m_1}}{b_1,b_2,\ldots,b_{m_2}}
\sim \binom{a_1+1,a_2+1,\ldots,a_{m_1}+1,0}{b_1+1,b_2+1,\ldots,b_{m_2}+1,0}.
\]
It is easy to see that the ranks and the defects are invariant in a similarity class.
Let $\Phi_{n,d}$ denote the set of similarity classes of symbols of rank $n$ and defect $d$.
Note that $\Phi_{n,d}=\emptyset$ if $\bigl\lfloor(\frac{d}{2})^2\bigr\rfloor>n$.

We define a mapping
\[
\binom{a_1,a_2,\ldots,a_{m_1}}{b_1,b_2,\ldots,b_{m_2}}\mapsto
\sqbinom{a_1-(m_1-1),a_2-(m_1-2),\ldots,a_{m_1-1}-1,a_{m_1}}{b_1-(m_2-1),b_2-(m_2-1),\ldots,b_{m_2-1}-1,b_{m_2}}
\]
from symbols to bipartitions.
It is easy to see that two symbols in the same similarity class has the same image.
So the above mapping will be regarded as a mapping from the set of similarity classes of symbols to the set of
bipartitions.
Moreover, it is not difficult to check that the above mapping gives a bijection
\begin{equation}\label{0305}
\Upsilon\colon\Phi_{n,d}\longrightarrow\calp_2\bigl(n-\bigl\lfloor\bigl(\tfrac{d}{2}\bigr)^2\bigr\rfloor\bigr).
\end{equation}
In particular, $|\Phi_{n,d}|=p_2\bigl(n-(\tfrac{d}{2})^2\bigr)$ if $d$ is even.
Note that if a symbol $\Lambda$ of rank $n$ is degenerate, then ${\rm def}(\Lambda)=0$ and $\Upsilon(\Lambda)$ is a degenerate
bipartition in $\calp_2(n)$.
Finally, we define
\begin{equation}\label{0306}
\Phi^+_n =\bigcup_{d\equiv 0\pmod 4}\Phi_{n,d},\qquad
\Phi^-_n =\bigcup_{d\equiv 2\pmod 4}\Phi_{n,d},\qquad
\Phi_n =\Phi^+_n\cup\Phi^-_n =\bigcup_{d\text{ even}}\Phi_{n,d}.
\end{equation}

\begin{exam}\label{0201}
For $n=4$ and $d$ even, we have the following
\begin{align*}
\Phi_{4,0} &=\textstyle\bigl\{\,\Lambda,\Lambda^\rmt\mid\Lambda=\binom{4}{0},\binom{3}{1},\binom{4,1}{1,0},\binom{3,2}{1,0},\binom{3,1}{2,0},
\binom{3,0}{2,1},\binom{4,2,1}{2,1,0},\binom{3,2,1}{3,1,0},\binom{4,3,2,1}{3,2,1,0}\bigr\}\cup\bigl\{\binom{2}{2},\binom{2,1}{2,1}\,\bigr\}, \\
\Phi_{4,2} &=\textstyle\bigl\{\binom{4,0}{-},\binom{3,1}{-},\binom{3,2,1}{0},\binom{4,1,0}{1},\binom{3,2,0}{1},
\binom{3,1,0}{2},\binom{2,1,0}{3},\binom{4,2,1,0}{2,1},\binom{3,2,1,0}{3,1},\binom{4,3,2,1,0}{3,2,1}\bigr\}, \\
\Phi_{4,4} &=\textstyle\bigl\{\binom{3,2,1,0}{-}\bigr\}, \\
\Phi_{4,-d} &=\textstyle\{\,\Lambda^\rmt\mid\Lambda\in\Phi_{4,d}\,\}, \\
\Phi^+_4 &=\Phi_{4,0}\cup\Phi_{4,4}\cup\Phi_{4,-4}, \\
\Phi^-_4 &=\Phi_{4,2}\cup\Phi_{4,-2}.
\end{align*}
Note that $|\Phi_{4,0}|=20=p_2(4)$, $|\Phi_{4,2}|=|\Phi_{4,-2}|=10=p_2(3)$, and $|\Phi_{4,4}|=|\Phi_{4,-4}|=1=p_2(0)$.
\end{exam}

\subsection{Families associated to special symbols}
A symbol $Z=\binom{a_1,a_2,\ldots,a_m}{b_1,b_2,\ldots,b_m}\in\Phi_{n,0}$ is called \emph{special}
if $a_1\geq b_1\geq a_2\geq b_2\geq\cdots\geq a_m\geq b_m$.
For a special symbol $Z$, let $Z_\rmI=Z\smallsetminus \binom{Z^*\cap Z_*}{Z^*\cap Z_*}$ be the subsymbol of ``singles'' in $Z$.
Clearly, ${\rm def}(Z_\rmI)=0$.
The \emph{degree} $\deg(Z)$ is defined to be $|(Z_\rmI)^*|=|(Z_\rmI)_*|$.
Note that a degenerate symbol is always special.
Moreover, a special symbol $Z$ with $\deg(Z)=0$ means that $Z$ is degenerate.
For a subsymbol $M\subset Z_\rmI$, we define
\[
\Lambda_M=(Z\smallsetminus M)\cup M^\rmt,
\]
i.e., $\Lambda_M$ is obtained from $Z$ by switching the row position of each entry of $M$.
It is clear that ${\rm def}(\Lambda_M)=-2\,{\rm def}(M)$.
For a special symbol $Z\in\Phi_{n,0}$, we define
\[
\Phi_Z=\{\,\Lambda_M\mid M\subset Z_\rmI\,\},
\]
i.e., $\Phi_Z$ is the set of symbols of exactly the same entries of $Z$.
It is clear that $|\Phi_Z|=2^{2\deg(Z)}$.
Moreover, we have
\begin{align}\label{0202}
\begin{split}
\Phi_Z\cap\Phi_{n,-2d} &=\{\,\Lambda_M\mid M\subset Z_\rmI,\ {\rm def}(M)=d\,\},\\
\Phi_n &=\bigcup_{Z\in\Phi_{n,0},\ \text{ special}}\Phi_Z.
\end{split}
\end{align}

\begin{exam}
We have the following special symbols $Z$ in $\Phi_{4,0}$:
\[
\begin{tabular}{c|c|cccccc|cc}
$Z$ & $\binom{3,1}{2,0}$ & $\binom{4}{0}$ & $\binom{3}{1}$ & $\binom{4,1}{1,0}$ & $\binom{4,2,1}{2,1,0}$ & $\binom{3,2,1}{3,1,0}$ & $\binom{4,3,2,1}{3,2,1,0}$ & $\binom{2}{2}$ & $\binom{2,1}{2,1}$ \\
\midrule
$Z_\rmI$ & $\binom{3,1}{2,0}$ & $\binom{4}{0}$ & $\binom{3}{1}$ & $\binom{4}{0}$ & $\binom{4}{0}$ & $\binom{2}{0}$ & $\binom{4}{0}$ & $\binom{-}{-}$ & $\binom{-}{-}$ \\
\midrule
$|\Phi_Z|$ & $16$ & $4$ & $4$ & $4$ & $4$ & $4$ & $4$ & $1$ & $1$ \\
\end{tabular}
\]
\end{exam}

\begin{exam}
Let $Z=\binom{3,1}{2,0}\in\Phi_{4,0}$.
Then $\deg(Z)=2$ and we have the following $16$ symbols $\Lambda_M$ in $\Phi_Z$:
\[
\begin{tabular}{c|cccccccc}
\toprule
$M$ & $\binom{-}{-}$ & $\binom{3}{-}$ & $\binom{1}{-}$ & $\binom{-}{2}$ & $\binom{-}{0}$ & $\binom{3,1}{-}$ & $\binom{-}{2,0}$ & $\binom{3}{2}$ \\
\midrule
$\Lambda_M$ & $\binom{3,1}{2,0}$ & $\binom{1}{3,2,0}$ & $\binom{3}{2,1,0}$ & $\binom{3,2,1}{0}$ & $\binom{3,1,0}{2}$ & $\binom{-}{3,2,1,0}$ & $\binom{3,2,1,0}{-}$ & $\binom{2,1}{3,0}$ \\
\midrule
\midrule
$M$ & $\binom{1}{0}$ & $\binom{3}{0}$ & $\binom{1}{2}$ & $\binom{3,1}{2}$ & $\binom{3,1}{0}$ & $\binom{3}{2,0}$ & $\binom{1}{2,0}$ & $\binom{3,1}{2,0}$ \\
\midrule
$\Lambda_M$ & $\binom{3,0}{2,1}$ & $\binom{1,0}{3,2}$ & $\binom{3,2}{1,0}$ & $\binom{2}{3,1,0}$ & $\binom{0}{3,2,1}$ & $\binom{2,1,0}{3}$ & $\binom{3,2,0}{1}$ & $\binom{2,0}{3,1}$  \\
\bottomrule
\end{tabular}
\]
\end{exam}

\subsection{Second proof of Theorem~\ref{0101}}

\begin{lem}\label{0301}
Let $Z\in\Phi_{n,0}$ be a special symbol,
and let $\delta=\deg(Z)$.
Then
\[
\left|\Phi_Z\cap\left(\bigcup_{d\equiv 0\pmod 4}\Phi_{n,d}\right)\right|
-\left|\Phi_Z\cap\left(\bigcup_{d\equiv 2\pmod 4}\Phi_{n,d}\right)\right|=\begin{cases}
1, & \text{if $\delta=0$};\\
0, & \text{if $\delta\geq 1$}.
\end{cases}
\]
\end{lem}
\begin{proof}
Note that $|(Z_\rmI)^*|+|(Z_\rmI)_*|=2\delta$,
and ${\rm def}(M)\equiv |M|\pmod 2$.
Then from (\ref{0202}), we have
\begin{align*}
\left|\Phi_Z\cap\left(\bigcup_{d\equiv 0\pmod 4}\Phi_{n,d}\right)\right|
&=|\{\,M\subset Z_\rmI\mid |M^*|+|M_*|\text{ even}\,\}|
=\sum_{k\geq 0,\text{ even}}C^{2\delta}_k, \\
\left|\Phi_Z\cap\left(\bigcup_{d\equiv 2\pmod 4}\Phi_{n,d}\right)\right|
&=|\{\,M\subset Z_\rmI\mid |M^*|+|M_*|\text{ odd}\,\}|
=\sum_{k\geq1,\text{ odd}}C^{2\delta}_k
\end{align*}
where $C^l_m=\frac{l!}{m!(l-m)!}$ is a binomial coefficient.
Note that by convention $C^l_m=0$ if $l,m\in\bbN\cup\{0\}$ and $l<m$.
It is well known from the binomial theorem that
\[
\sum_{k\geq 0,\text{ even}}C^{2\delta}_k-\sum_{k\geq1, \text{ odd}}C^{2\delta}_k=\begin{cases}
1, & \text{if $\delta=0$};\\
0, & \text{if $\delta\geq 1$}.
\end{cases}
\]
Then the lemma is proved.
\end{proof}

\begin{cor}\label{0304}
Let $n$ be a nonnegative integer.
Then $|\Phi_n^+|-|\Phi^-_n|=p(\tfrac{n}{2})$.
\end{cor}
\begin{proof}
From (\ref{0306}) and (\ref{0202}), we have
\begin{align*}
\Phi^+_n=\bigcup_{Z\in\Phi_{n,0},\text{ special}}\left(\Phi_Z\cap\bigcup_{d\equiv 0\pmod 4}\Phi_{n,d}\right),\\
\Phi^-_n=\bigcup_{Z\in\Phi_{n,0},\text{ special}}\left(\Phi_Z\cap\bigcup_{d\equiv 2\pmod 4}\Phi_{n,d}\right).
\end{align*}
Then by Lemma~\ref{0301}, we have
$|\Phi_n^+|-|\Phi^-_n|=\sum_{Z\in\Phi_{n,0},\ \deg(Z)=0} 1=p(\tfrac{n}{2})$.
\end{proof}

\begin{proof}[Second Proof of Theorem~\ref{0101}]
From (\ref{0305}),
we know that $|\Phi_{n,d}|=p_2\bigl(n-(\tfrac{d}{2})^2\bigr)$,
and from (\ref{0306}), we have
\[
|\Phi^+_n| =\sum_{k\in\bbZ,\text{ even}} p_2\bigl(n-k^2\bigr),\qquad
|\Phi^-_n| =\sum_{k\in\bbZ,\text{ odd}} p_2\bigl(n-k^2\bigr).
\]
Then the theorem follows from Corollary~\ref{0304} immediately.
\end{proof}

\begin{rem}\label{0302}
Let $\bfF_q$ denote the finite field of $q$ elements where $q$ is a power of an odd prime.
From Lusztig's theory, for $\epsilon\in\{+,-\}$, the set $\cale(\rmO^\epsilon_{2n}(\bfF_q),1)$
of unipotent characters of the finite even orthogonal group $\rmO^\epsilon_{2n}(\bfF_q)$
is parametrized by the set $\Phi^\epsilon_n$ (\cf.~\cite{lg} theorem 8.2, see also \cite{pan-finite-unipotent} \S 3.3).
Hence Corollary~\ref{0304} means that $\rmO^+_{2n}(\bfF_q)$ has $p(\frac{n}{2})$ more unipotent characters than
$\rmO^-_{2n}(\bfF_q)$.
\end{rem}

\begin{rem}\label{0303}
Another consequence of Theorem~\ref{0101} is that the number of quadratic unipotent conjugacy classes of
$\rmO^\epsilon_{2n}(\bfF_q)$ is equal to the number of quadratic unipotent characters of
$\rmO^\epsilon_{2n}(\bfF_q)$ for $\epsilon\in\{+,-\}$.
The detail will appear in a paper in preparation by the second author.
\end{rem}

\appendix

\section{A Congruence Relation for $p_2(n)$}

The following theorem on a congruence relation for $p_2(n)$ has been proved in \cite{HL} and \cite{CJW}.
Below we provide another proof along the approach taken by Michael Hirschhorn~\cite{Hirschhorn}.

\begin{thm}\label{thm:bi-congruence}
We have
\[
p_2(5n+2)\equiv p_2(5n+3)\equiv p_2(5n+4)\equiv 0\pmod 5
\]
for any nonnegative integer $n$.
\end{thm}

Before the proof, we first set up some notations.
Let
\[
(a;q)= \prod_{k=0}^\infty (1 - aq^k),
\]
and define
\[
R(q) = \frac{ (q^2; q^5) (q^3; q^5) }{ (q; q^5) (q^4; q^5) }\qquad\text{and}\qquad
c = R(q^5).
\]
Then Hirschhorn obtained (see also \cite{aigner2007course} p.140)
\begin{equation}\label{eq:partition-modulo}
\sum_{n=0}^\infty p(n) q^n
= \frac{ (q^{25}; q^{25})^5 }{ (q^5; q^5)^6 } f(q),
\end{equation}
where
\[
f(q) = c^4 + c^3 q + 2c^2 q^2 + 3c q^3 + 5q^4 - 3 c^{-1} q^5 + 2 c^{-2} q^6 - c^{-3} q^7 + c^{-4}q^8.
\]
From (\ref{eq:partition-modulo}) and the fact that the expansion of $c$ only involves powers of $q^5$,
Hirschhorn obtained
\[
\sum_{n=0}^\infty p(5n+4)q^{5n+4}
=5\frac{(q^{25};q^{25})^5}{(q^5;q^5)^6}q^4
\]
and concluded the famous congruence by Ramanujan:
$p(5n+4) \equiv 0 \pmod{5}$.

\begin{proof}[Proof of Theorem~\ref{thm:bi-congruence}]
We square both sides of (\ref{eq:partition-modulo}) and reach
\begin{equation}\label{eq:bipartition-modulo}
\sum_{n=0}^\infty p_2(n) q^n
= \left( \sum_{n \geqslant 0} p(n) q^n \right)^2
= \frac{ (q^{25}; q^{25})^{10} }{ (q^5; q^5)^{12} } f(q)^2,
\end{equation}
and
\begin{multline*}
f(q)^2 = c^{8} + 2 c^{7} q + 5  c^{6} q^{2} + 10 c^{5} q^{3} + 20 c^{4} q^{4} + 16 c^{3} q^{5} \\
+ 27 c^{2} q^{6} + 20 c q^{7} + 15 q^{8} - 20 c^{-1} q^{9} + 27 c^{-2} q^{10} - 16 c^{-3} q^{11} \\
+ 20 c^{-4} q^{12} - 10 c^{-5} q^{13} + 5 c^{-6} q^{14} - 2 c^{-7} q^{15} + c^{-8} q^{16}.
\end{multline*}
Using the fact again that the expansion of $c$ only involves powers of $q^5$,
we obtain
\begin{align*}
\sum_{n=0}^\infty p_2(5n+2) q^{5n+2}
&= 5 \frac{ (q^{25}; q^{25})^{10} }{ (q^5; q^5)^{12} } \left( c^6 q^2 + 4 c q^7 + 4 c^{-4} q^{12} \right), \\
\sum_{n=0}^\infty p_2(5n+3) q^{5n+3}
&= 5 \frac{ (q^{25}; q^{25})^{10} }{ (q^5; q^5)^{12} } \left( 2 c^5 q^3 + 3 q^8 - 2 c^{-5} q^{13} \right), \\
\sum_{n=0}^\infty p_2(5n+4) q^{5n+4}
&= 5 \frac{ (q^{25}; q^{25})^{10} }{ (q^5; q^5)^{12} } \left( 4 c^4 q^4 - 4 c^{-1} q^9 + c^{-6} q^{14} \right).
\end{align*}
The theorem now follows from these identities immediately.
\end{proof}

\bibliography{refer}

\providecommand{\bysame}{\leavevmode\hbox to3em{\hrulefill}\thinspace}
\providecommand{\MR}{\relax\ifhmode\unskip\space\fi MR }
\providecommand{\MRhref}[2]{%
  \href{http://www.ams.org/mathscinet-getitem?mr=#1}{#2}
}
\providecommand{\href}[2]{#2}
\begin{thebibliography}{CJW06}

\bibitem[Aig07]{aigner2007course}
Martin Aigner, \emph{A {C}ourse in {E}numeration}, Grad. Texts in Math., vol.
  238, Springer, Berlin, 2007.

\bibitem[And76]{andrews}
George~E. Andrews, \emph{The {T}heory of {P}artitions}, Encyclopedia Math.
  Appl., vol.~2, Addison-Wesley Publishing Co., Massachusetts, 1976.

\bibitem[CJW06]{CJW}
William Y.~C. Chen, Kathy~Q. Ji, and Herbert~S. Wilf, \emph{{BG}-ranks and
  $2$-cores}, Electron. J. Combin. \textbf{13} (2006), Note 18, 5 pp.

\bibitem[GP00]{Geck-Pfeiffer}
Meinolf Geck and G\"otz Pfeiffer, \emph{Characters of {F}inite {C}oxeter
  {G}roups and {I}wahori-{H}ecke algebras}, London Math. Soc. Monogr. (N.S.),
  no.~21, The Clarendon Press, Oxford University Press, New York, 2000.

\bibitem[Hir00]{Hirschhorn}
Michael~D. Hirschhorn, \emph{An identity of {R}amanujan, and applications},
  $q$-series from a {C}ontemporary {P}erspective, Contemp. Math., no. 254,
  American Mathematical Society, Providence, RI, 2000, pp.~229--234.

\bibitem[HL04]{HL}
Paul Hammond and Richard Lewis, \emph{Congruences in ordered pairs of
  partitions}, Int. J. Math. Math. Sci. (2004), no.~45--48, 2509--2512.

\bibitem[Lus77]{lg}
George Lusztig, \emph{Irreducible representations of finite classical groups},
  Invent. Math. \textbf{43} (1977), 125--175.

\bibitem[Pan24]{pan-finite-unipotent}
S.-Y. Pan, \emph{Howe correspondence of unipotent characters for a finite
  symplectic/even-orthogonal dual pair}, Amer. J. Math. \textbf{146} (2024),
  no.~3, 813--869.

\end{thebibliography}
\bibliographystyle{amsalpha}

\end{document}